\documentclass[11pt,reqno,tbtags,a4paper]{amsart}
\usepackage{amssymb}
\usepackage{tikz}
\usepackage{pgfplots}
\pgfplotsset{compat=1.10}
\usetikzlibrary{shapes.geometric,arrows,fit,matrix,positioning}
\tikzset
{
    treenode/.style = {circle, draw=black, align=center, minimum size=1cm},
    subtree/.style  = {isosceles triangle, draw=black, align=center, minimum height=0.5cm, minimum width=1cm, shape border rotate=90, anchor=north}
}

\usepackage{url}
\usepackage[square,numbers]{natbib}
\bibpunct[, ]{[}{]}{;}{n}{,}{,}

\title[Continuous time digital search tree]
{Continuous time digital search tree\\
and a border aggregation model}

\date{
28 April, 2020}

\author{Svante Janson} 
\thanks{Partly supported by the Knut and Alice Wallenberg Foundation}
\address[Svante Janson]{Department of Mathematics, Uppsala University, PO
  Box 480, SE-751~06 Uppsala, Sweden}
\email{svante.janson@math.uu.se}
\urladdr{http://www.math.uu.se/svante-janson}

\author{Debleena Thacker}  
\address[Debleena Thacker]
{Department of Mathematics, NYU, Shanghai,
         1555 Century Avenue, Pudong New District, Shanghai, China 200122
}         
\email{thackerdebleena@gmail.com}

\subjclass[2010]{} 

\numberwithin{equation}{section}

\renewcommand\le{\leqslant}
\renewcommand\ge{\geqslant}

\allowdisplaybreaks

\theoremstyle{plain}
\newtheorem{theorem}{Theorem}[section]

\newtheorem{corollary}[theorem]{Corollary}
\newtheorem{conjecture}[theorem]{Conjecture}

\theoremstyle{definition}

\newtheorem{definition}[theorem]{Definition}

\newtheorem{remark}[theorem]{Remark}

\theoremstyle{remark}

\newenvironment{romenumerate}[1][-10pt]{
\addtolength{\leftmargini}{#1}\begin{enumerate}
 \renewcommand{\labelenumi}{\textup{(\roman{enumi})}}%
 }{\end{enumerate}}

\newcounter{oldenumi}
{\setcounter{oldenumi}{\value{enumi}}
\begin{romenumerate} \setcounter{enumi}{\value{oldenumi}}}
{\end{romenumerate}}

\newcounter{thmenumerate}

\newcounter{xenumerate}   

\newcommand\pfitemx[1]{\par#1:}
\newcommand\pfitemref[1]{\pfitemx{\ref{#1}}}

\newcommand{\refT}[1]{Theorem~\ref{#1}}
\newcommand{\refTs}[1]{Theorems~\ref{#1}}
\newcommand{\refC}[1]{Corollary~\ref{#1}}

\newcommand{\refR}[1]{Remark~\ref{#1}}
\newcommand{\refS}[1]{Section~\ref{#1}}
\newcommand{\refSS}[1]{Section~\ref{#1}}

\newcommand{\refD}[1]{Definition~\ref{#1}}

\newcommand{\refConj}[1]{Conjecture~\ref{#1}}

\begingroup
  \divide\count255 by 60
  \multiply\count255 by -60
  \advance\count255 by \time
  \ifnum \count255 < 10 \xdef\klockan{\the\count1.0\the\count255}
  \else\xdef\klockan{\the\count1.\the\count255}\fi
\endgroup

\newcommand\set[1]{\ensuremath{\{#1\}}}
\newcommand\bigset[1]{\ensuremath{\bigl\{#1\bigr\}}}

\newcommand\xpar[1]{(#1)}
\newcommand\bigpar[1]{\bigl(#1\bigr)}
\newcommand\Bigpar[1]{\Bigl(#1\Bigr)}
\newcommand\biggpar[1]{\biggl(#1\biggr)}
\newcommand\lrpar[1]{\left(#1\right)}

\newcommand\xcpar[1]{\{#1\}}

\def\rompar(#1){\textup(#1\textup)}    
\newcommand\xfrac[2]{#1/#2}

\newcommand\parfrac[2]{\lrpar{\frac{#1}{#2}}}

\newcommand\Bigparfrac[2]{\Bigpar{\frac{#1}{#2}}}

\def\xexp(#1){e^{#1}}

\newcommand\floor[1]{\lfloor#1\rfloor}

\newcommand\Ktoo{\ensuremath{{K\to\infty}}}

\newcommand\punkt{.\spacefactor=1000}    
\newcommand\iid{i.i.d\punkt}    
\newcommand\ie{i.e\punkt}
\newcommand\eg{e.g\punkt}
\newcommand\viz{viz\punkt}

\newcommand\whp{w.h.p\punkt}

\newcommand{\tend}{\longrightarrow}

\newcommand\pto{\overset{\mathrm{p}}{\tend}}

\newcommand\eqd{\overset{\mathrm{d}}{=}}

\newcommand\bbZ{\mathbb Z}

\newcounter{CC}
\newcounter{cc}

\newcommand\E{\operatorname{\mathbb E{}}}
\renewcommand\P{\operatorname{\mathbb P{}}}

\newcommand\Exp{\operatorname{Exp}}
\newcommand\Po{\operatorname{Po}}

\newcommand\Be{\operatorname{Be}}

\newcommand\gl{\lambda}

\newcommand\go{\omega}

\newcommand\gth{\theta}
\newcommand\eps{\varepsilon}

\renewcommand\phi{\xxx}  

\newcommand\cA{\mathcal A}

\newcommand\cO{\mathcal O}

\newcommand\ett[1]{\boldsymbol1\xcpar{#1}}

\newcommand\qw{^{-1}}

\newcommand\qqw{^{-1/2}}

\newcommand\intoo{\int_0^\infty}

\newcommand\ooo{[0,\infty)}

\newcommand\dd{\,\mathrm{d}}

\newcommand{\pgf}{probability generating function}

\newcommand\lhs{left-hand side}
\newcommand\rhs{right-hand side}

\newcommand\etto{\bigpar{1+o(1)}}

\newcommand\dt{\mathcal T}
\newcommand\ct{\mathfrak T}
\newcommand\dtn{\dt_n}
\newcommand\ctt{\ct_t}
\newcommand\too{T_\infty}
\newcommand\Voo{V(\too)}
\newcommand\Ve{V_{\mathsf e}}
\newcommand\Vi{V_{\mathsf i}}
\newcommand\xoo{_0^\infty}
\newcommand\taux[1]{\tau(#1)}
\newcommand\taun{\taux{n}}

\newcommand\cxi{\Xi}
\newcommand\YY{Y^*}
\newcommand\YYY{Y^{*\prime}}
\newcommand\YYX{Y^{**}}
\newcommand\oL{o_{\mathsf L}}
\newcommand\oR{o_{\mathsf R}}
\newcommand\xxi{\Xi_{K+1}^*}
\newcommand\xil{\Xi_{\mathsf L}}
\newcommand\xir{\Xi_{\mathsf R}}
\newcommand\YYL{\YY_{\mathsf L}}
\newcommand\YYR{\YY_{\mathsf R}}
\newcommand\he{h_{\mathsf e}}

\newcommand\dst{digital search tree}
\newcommand\cdst{continuous-time digital search tree}
\newcommand\bam{border aggregation model}
\newcommand\cbam{continuous-time \bam}
\newcommand\OO{\cO}
\newcommand\oo{o}

\newcommand\Op{\OO_{\mathrm p}}
\newcommand\lgg{\log_2}
\newcommand\kkk{\tilde{k}}
\newcommand\hth{\tilde\gth}
\renewcommand\ln{\log}
\newcommand\mk{m_K}
\newcommand\nk{\bar n_K}

\begin{document}

\begin{abstract} 
We consider the continuous-time version of the random digital search tree, 
and construct a coupling with a border aggregation model as studied in Thacker and Volkov (2018), showing 
a relation between the height of the tree and the time required for
aggregation.
This relation carries over to the corresponding discrete-time models.
As a consequence we find a very precise asymptotic result for the time to
aggregation, using recent results by Drmota et al.\ (2020) 
for the  digital search tree. 
\end{abstract}

\maketitle


\section{Introduction}\label{S:intro}

A \emph{digital search tree} $\dtn$ is a binary tree constructed from a
sequence of $n$ binary strings (called \emph{items} or \emph{keys}).
(See \refS{Sprel} for details, as well as for definitions
of other concepts used below.)
We consider here only the case when the items are 
\iid{} (independent, identically distributed) random infinite binary strings, 
and furthermore, in each string the
digits are independent 
$\Be(1/2)$ random variables,
\ie, 0 or 1 with probability $\frac12$ each.
(See \refS{Sdary} for the $b$-ary case.)
Digital search trees are among the fundamental objects of study in computer
science algorithms and have been studied by many authors,
see \eg{} \cite{Al_Sh_1988, Drmota2002, Drmota-book,
DrmotaEtAl-DST, Dr_Ja_Ne_2008, JacquetS-book,
Kn_Sz_2000, KnuthIII, Mahmoud-book}.

Our main concern is with a continuous-time version of the digital search
tree, studied also by \citet{Al_Sh_1988}.
This can be defined by assuming that an infinite sequence $(W_n)$ of items
arrive at random times that are given by a Poisson process; 
we then let $\ctt$ be the digital search tree defined by the strings arriving up
to time $t$.
The continuous-time version is thus a Poissonization of the standard version.
A simple but central result
(\refT{T=} and \cite{Al_Sh_1988})
is that the \cdst{}
$\ctt$ also can be defined in two
other ways that turn out to be equivalent;
in particular, the continuous-time digital search tree is equivalent to
\emph{first-passage percolation} on the infinite binary tree, with the
passage times of the edges exponentially distributed such that the
passage time of an edge between nodes of depth $k-1$ and $k$ has expectation
$2^k$.

 Our main result couples the \cdst{}
and a \emph{border aggregation model} on a binary tree
studied by \citet{ThackerVolkov}.
In this model, we fix $K\ge1$ and consider the complete binary tree $T_K$ of
height $K$. 
We recursively define a collection of randomly growing
subset of \emph{sticky} nodes $S_n$, such that $S_0$ is the set of the $2^K$
nodes of depth $K$. $S_n$ is obtained from $S_{n-1}$ as follows: 
A particle is released from the root, and performs a (directed) random walk 
until it comes to a neighbour $v_n$ of $S_{n-1}$. The random walk now stops, and
the node $v_n$  becomes "sticky"; in other words, $S_n:=S_{n-1}\cup\set{v_n}$.
This is repeated until the root $o$ is sticky. 
Let $\xi_K$ be the random number of particles to be
released until the root $o$ is sticky.
We define also a continuous-time version of the \bam{} by assuming that 
particles start from the root at times given by a Poisson process
(and that the random walk itself takes no time); let $\Xi_K$ be the random
time that the root gets sticky in the \cbam.

\begin{figure}[h]
\begin{center}
 \begin{tikzpicture}[scale=0.8, every node/.style={circle, inner sep=2.35pt, draw=black, fill=},
 level 1/.style={sibling distance=7cm, fill=white},
        level 2/.style={sibling distance=4cm, },
        level 3/.style={sibling distance=1.9cm, },
        level 4/.style={sibling distance=.5cm, fill=red}]
        \node()[fill=white]{}
        child{node (0) [fill=white]{}
            child{node (00) {}
                child{node (000) {}
                    child{node (0000){}}
                    child{node(0001){}}}
                child{node (001){}
                    child{node (0010){}}
                    child{node(0011){}}}}
            child{node (01) {}
                child{node(010){}
                    child{node(0100){}}
                    child{node(0101){}}}    
                child{node(011){}
                    child{node(0110){}}
                    child{node(0111){}}}
            }}
        child{node (1) {}
            child{node (10) {}
                child{node (100) {}
                    child{node(1000){}}
                    child{node(1001){}}}
                child{node (101){}
                    child{node (1010) {}}
                    child{node (1011) {}}}}
            child{node (11) {}
                child{node(110){}
                    child{node(1100){}}
                    child{node(1101){}}}
                child{node(111){}
                    child{node(1110){}}
                    child{node(1111){}}}}};
   \end{tikzpicture} 
\caption{Binary tree $T_K$ with $K=4$, and the red nodes denoting $S_0$}
\label{Fig: Binary_tree_boundary}
\end{center}	
\end{figure}
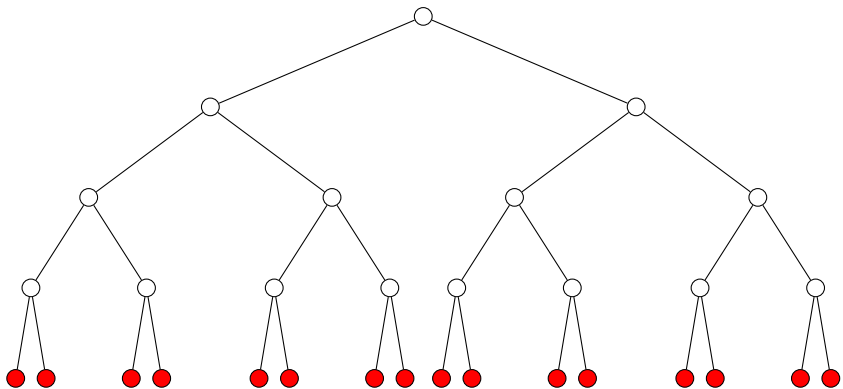

Note that the \dst{} and the \bam{} grow in opposite directions:
the \dst{} grows from the root downwards, while the \bam{} grows from the 
starting boundary at depth $K$ up towards the root. Nevertheless, they are
connected by a kind of duality, and we show 
that the time $\xi_K$ or $\Xi_K$ taken by the \bam{} equals in distribution
the time the (discrete or continuous-time, respectively)
\dst{} reaches (external) height $K$ (\refT{TB}).
Equivalently, we have the following results,
where $\he(T)$ denotes the external height of a  tree $T$.
\begin{theorem}\label{TC}
The following equalities hold.
  \begin{romenumerate}
  \item \label{Tcd}
(Discrete time.)
For any $K\ge1$ and $n\ge0$,
  \begin{equation}\label{tcd}
  \P\bigpar{\xi_{K}\le n}
=
\P\bigpar{\he(\dtn)\ge K}.
\end{equation}

  \item\label{Tcc} (Continuous time.)
For any $K\ge1$ and $t\ge0$,
  \begin{equation}\label{tcc}
  \P\bigpar{\Xi_{K}\le t}
=
\P\bigpar{\he(\ctt)\ge K}.
\end{equation}
  \end{romenumerate}
\end{theorem}
We show this
using the continuous-time versions; the result then easily
transfers to discrete time too.

Asymptotic properties
of the height  $\he(\dtn)$
of digital search trees have been studied by several authors
\cite{Al_Sh_1988, Drmota2002, DrmotaEtAl-DST,Kn_Sz_2000}. 
In particular, very precise results are proved by
\citet{DrmotaEtAl-DST}. We use these results and \refT{TC} 
to obtain the following result on the distribution of $\xi_K$, which
improves on 
the bounds
$2^{K-2\sqrt{K}+\OO(K^{-1/2})}\le\xi_K \le 2^{K-1+\oo(1)}$ \whp{}
shown in \cite[Theorem 5]{ThackerVolkov}.
\begin{theorem}\label{Txi} As \Ktoo,
  \begin{equation}\label{txi}
\log_2  \xi_K = K -\sqrt{2K} + \frac12\log_2 K -\frac{1}{\ln 2} +
\frac{\log_2K}{4\sqrt{2K}} + \Op\Bigparfrac{1}{\sqrt K}.
  \end{equation}
\end{theorem}

For convenience, let
  \begin{equation}\label{mk}
\mk =2^{ K -\sqrt{2K} + \frac12\log_2 K -\frac{1}{\ln2} +
\frac{\log_2K}{4\sqrt{2K}}}.
  \end{equation}
Then, \refT{Txi} says that
$
\log_2\xi_K=\log_2 \mk+ \Op\bigpar{1/\sqrt K}
$,
or, equivalently,
\begin{align}\label{xik2}
  \xi_K = \mk\bigpar{1+\Op\bigpar{K\qqw}}.
\end{align}
\begin{conjecture}\label{CONJE}
  We conjecture that also
\begin{align}\label{conjE}
\E \xi_K 
 = \mk\bigpar{1+\OO\bigpar{K\qqw}}
=2^{ K -\sqrt{2K} + \frac12\log_2 K -\frac{1}{\ln2} +
\frac{\log_2K}{4\sqrt{2K}} + \OO\parfrac{1}{\sqrt K}}.
\end{align}
\end{conjecture}
We have not been able to prove \eqref{conjE}, see \refR{RE}, but
as a corollary of \refT{Txi} and tail estimates by \citet{Drmota2002},
we show the following cruder estimate.
\begin{theorem}\label{TE}
As \Ktoo,
  \begin{align}
      \E\xi_K=\E\Xi_K= \etto\mk 
=2^{ K -\sqrt{2K} + \frac12\log_2 K -\frac{1}{\ln2} +o(1)}.
  \end{align}
\end{theorem}


The \bam{} was introduced as \emph{internal erosion} by  \citet{Le_Pe_2007}.  
In \cite{ThackerVolkov}, the border aggregation model was studied
on a variety of graphs, and several interesting results were obtained.
One reason of the interest in the \bam{} is its
possible connections to other interesting models in statistical
physics, in particular the classical 
\emph{diffusion limited aggregation} (DLA)
\cite{Wi_Sa_1983,Ke_1987, Ba_Pe_Pe_1997}, and
\emph{internal diffusion limited aggregation} (IDLA)
\cite{LawlerBG, Le_Si_2018, Si_2019}. 
It is conjectured \cite{Le_Pe_2007} that
on $\bbZ^2$     DLA and the border aggregation model are
"inversions" of each other in some sense; however, no rigorous results are
known. 
(Nevertheless, \cite{ThackerVolkov} uses
bounds obtained in \cite{Ke_1987} for DLA to obtain
 results for the border aggregation model.)
Note that
the \dst{} can be regarded as IDLA on the infinite binary tree
(see \refS{SSdst});
moreover, it can also be regarded as DLA on the same infinite binary tree,
see \citet[Lemma 1.3]{Ba_Pe_Pe_1997}. 
Thus, our results show a connection between the \bam{} and IDLA or DLA
on  trees. (Note that on $\bbZ^d$, IDLA is very different
from DLA and the \bam, with asymptotically a round shape
\cite{LawlerBG}.)

The rest of the paper is organized as follows.
\refS{Sprel} contains definitions and other preliminaries.
\refS{Scdst+} gives the equivalence of the 
different constructions of the \cdst. 
\refS{STB} contains the coupling of the \dst{} and the \bam,  
leading to the proof of \refT{TC},
and then 
\refS{SpfTxi} gives the proofs of \refTs{Txi} and \ref{TE}.
\refS{Sdary} discusses briefly extensions to the $b$-ary case.

\section{Preliminaries}\label{Sprel}

We recall some standard notation, adding some perhaps less standard details.

\subsection{General}

 
$\Exp(\gl)$ denotes an exponential distribution with \emph{rate} $\gl$,
i.e., with the density function $\gl e^{-\gl x}$, $x>0$, 
and thus the expectation  $1/\gl$. 

$\Op(a_n)$, where $a_n$ is a given positive sequence, denotes some sequence
of random variables $X_n$ such that the family$\set{X_n/a_n}$ is bounded in
probability, \ie,
$\lim_{C\to\infty}\sup_n\P(|X_n|>C a_n)=0$.

$\go(1)$ denotes a sequence tending to $+\infty$.

$x\land y$ denotes $\min\set{x,y}$.

\subsection{Binary trees}

An \emph{(extended) binary tree} is a rooted tree where each node has either 0
or two children; in the latter case there is one left child and one right
child.
Nodes with 0 children (leaves) are called \emph{external nodes} and nodes
with 2 children are called \emph{internal nodes}.

Let $\Vi(T)$ denote the set of internal nodes of $T$, and $\Ve(T)$ the set of
external nodes.

The root of a binary tree is denoted $o$.
The depth $d(v)$ of a node in a binary tree is the distance from $v$ to the
root $o$; thus $d(o)= 0$.

If $v$ and $w$ are nodes in a binary tree $T$, 
then $v\preceq w$ means that $v$
is on the path from the root to $w$ (including the endpoints).

Unless we say otherwise, we consider only finite binary trees.
However, we let $\too$ denote the infinite binary tree where each node
has two children. 
Thus, $\too$ has
$2^k$ nodes of depth $k$, $k\ge0$.
Every finite binary tree can be regarded as a subtree of $\too$.

The size $|T|:=|\Vi(T)|$ of an extended binary tree is the number of internal
nodes. 
Thus an extended binary tree of size $n$ has $n$ internal and $n+1$ external
nodes.

A binary tree is \emph{empty} if it has size 0, i.e., if there is no
internal node and only a single external node (the root).

The (external) 
height $\he(T)$ of a binary tree $T$ is the maximum depth of an 
external
node, \ie, (with $\max\emptyset:=-1$ for the empty tree)
\begin{equation}
  \he(T):=\max\bigset{d(v):v\in \Ve(T)}
=\max\bigset{d(w):w\in \Vi(T)}+1.
\end{equation}

$T_K$ is the complete binary tree of height $K$; it has 
$2^K$ external nodes, all at depth $K$, and thus
$2^K-1$ internal nodes.

\begin{remark}
It is also common to study binary trees without external nodes;
we may call them \emph{reduced binary tree}. 
The subtree of internal nodes in an extended binary tree is a reduced binary
tree (including the case of a reduced empty tree with no nodes),
and this gives an obvious $1$--$1$ correspondence between
extended and reduced binary trees.
In the present papers, all binary trees are extended binary trees as defined
above.
\end{remark}

\subsection{A random walk}\label{SSrandomwalk}

Given an extended binary  tree $T$, consider the random walk defined
by starting at the root, and then moving repeatedly from the current node to
one of its children, chosen at random with probability $1/2$ each
(independently of previous choices), until we reach an external node.

For an external node $v$, 
let $p_v$ be the probability that this random walk ends in $v$.
Thus $(p_v)_{v\in\Ve(T)}$ is a probability distribution on $\Ve$, which we call the
\emph{harmonic measure} of $T$.
Obviously, the harmonic measure is given by
\begin{equation}\label{harmonic}
  p_v=2^{-d(v)},
\qquad v\in \Ve(T).
\end{equation}

By construction, the harmonic measure is a probability measure, and thus,
for any finite binary tree $T$,
\begin{equation}\label{harmsum}
  \sum_{v\in\Ve(T)}2^{-d(v)}=1.
\end{equation}
(Alternatively, \eqref{harmsum} is easily seen by induction on the size $|T|$.)

\subsection{Boundaries}\label{SSboundary}

We say that a finite
set $B$ of nodes in $\Vi(\too)$ is a \emph{boundary}, if every
infinite path from the root contains exactly one element of $B$.
The set of external nodes $\Ve(T)$ of a finite binary tree is a boundary;
conversely, given a boundary $B$, there exists exactly one binary tree $T$ with
$B=\Ve(T)$. (The internal nodes of $T$ are the nodes $v$ that are strict 
ancestors of some node $w\in B$.) 
Hence there is a $1$--$1$ correspondence between
(finite) binary trees and boundaries, given by $T\leftrightarrow\Ve(T)$. 

Given a boundary $B$, the harmonic measure \eqref{harmonic}
on the corresponding tree is a
probability measure on $B$, which we also call the 
\emph{harmonic measure on  $B$}.


\subsection{Digital search trees}\label{SSdst}

A \emph{digital search tree} is a binary tree 
constructed recursively from
a sequence of $n\ge0$ infinite binary strings $W_1,\dots,W_n$ 
(called \emph{items}) as follows; 
the digital search tree has size $n$ and each
internal node stores one of the items.
See \eg{}
\cite[Section 6.3]{KnuthIII},
\cite[Section 6.1]{Mahmoud-book},
\cite[Section 1.4.3]{Drmota-book},
\cite[Section 6.4]{JacquetS-book}.

\begin{definition}\label{Ddtn0}
The digital search tree is constructed as follows.
  \begin{romenumerate}
\item 
Start with an empty binary tree, containing only the root as an external
node.
\item 
The items $W_i$ arrive one by one, in order;
each item comes first to the root of the tree.
\item 
When an item comes to an external node, it is stored there.
The node becomes internal and two new
external nodes are added as children to it.
\item 
When an item $W_i$
comes to an internal node $v$ at depth $d$, it is passed to the 
left [right]
child of $v$ if the $(d+1)$th bit of $W$ is 0 [1].
The construction proceeds recursively until an external node is reached.
\end{romenumerate}
\end{definition}


We shall only consider the random case, where
each string $W_i$ is a random string of independent bits, each with the
symmetric $\Be(1/2)$ distribution, and furthermore the strings are
independent.
We let $\dtn$ denote the random digital search tree constructed from such
strings, and we consider the sequence $(\dtn)\xoo$ constructed from an
infinite sequence of items $(W_n)_1^\infty$.

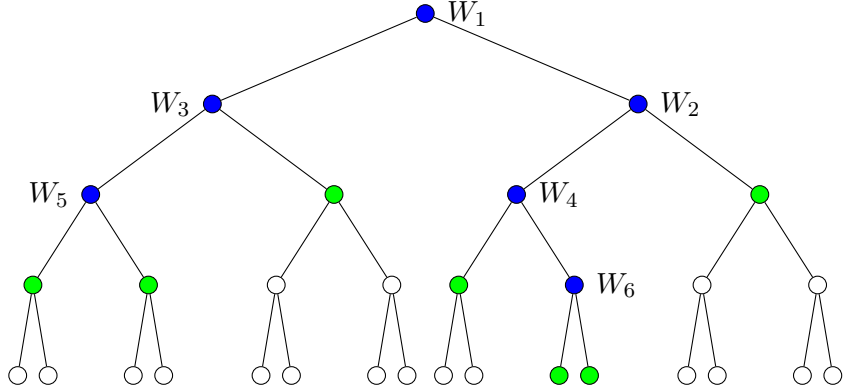
\begin{figure}[h]
\begin{center}
 \begin{tikzpicture}[scale=0.8, every node/.style={circle, inner sep=2.35pt, draw=black, fill=,},
 level 1/.style={sibling distance=7cm, fill=white,},
        level 2/.style={sibling distance=4cm, },
        level 3/.style={sibling distance=1.9cm, },
        level 4/.style={sibling distance=.5cm}]
        \node()[fill=blue, label=right:$W_1$]{}
        child{node (0) [fill=blue, label=left:$W_3$]{}
            child{node (00)[fill=blue, label=left:$W_5$] {}
                child{node (000)[fill=green] {}
                    child{node (0000){}}
                    child{node(0001){}}}
                child{node (001)[fill=green]{}
                    child{node (0010){}}
                    child{node(0011){}}}}
            child{node (01)[fill=green] {}
                child{node(010){}
                    child{node(0100){}}
                    child{node(0101){}}}    
                child{node(011){}
                    child{node(0110){}}
                    child{node(0111){}}}
            }}
        child{node (1)[fill=blue, label=right:$W_2$] {}
            child{node (10)[fill=blue, label=right:$W_4$] {}
                child{node (100)[fill=green] {}
                    child{node(1000){}}
                    child{node(1001){}}}
                child{node (101)[fill=blue, label=right:$W_6$]{}
                    child{node (1010)[fill=green] {}}
                    child{node (1011)[fill=green] {}}}}
            child{node (11)[fill=green] {}
                child{node(110){}
                    child{node(1100){}}
                    child{node(1101){}}}
                child{node(111){}
                    child{node(1110){}}
                    child{node(1111){}}}}};					
\end{tikzpicture} 
\caption{
Digital search tree for $6$ items; $W_1= \{01011\ldots\}$, $W_2=\{10011\ldots\}, W_3=\{00101\ldots\}, W_4=\{10110\ldots\},W_5=\{00011\ldots\}, W_6=\{10100\ldots\}$. The green nodes are the external nodes and the blue nodes are the internal nodes. 
}\label{Fig: Digital_search_tree}
\end{center}	
\end{figure}

It is obvious from the definitions, that when constructing the random
digital search tree $\dtn$,
the $i$th string $W_i$ performs a random walk on $\dt_{i-1}$
as described in \refSS{SSrandomwalk}.
(Hence, the \dst{} equals IDLA for this directed random walk, as said in the
introduction.)
Consequently, the sequence of random digital search trees $(\dtn)\xoo$ can
also be
defined as follows, without explicitly using random strings.

\begin{definition}\label{Ddtn}
  The random digital search trees $\dtn$, $n\ge0$, 
are constructed recursively, starting with $\dt_0$
  empty.
$\dt_{n+1}$ is obtained from $\dtn$ by choosing an external node $v$ in $\dt_n$
at random according to the harmonic measure \eqref{harmonic} and converting
this node $v$ to an internal node by adding two (external) children to it. 
\end{definition}

\subsection{Continuous-time digital search trees}\label{SScdst}

We think of item $W_i$ as arriving at time $i$, and $(\dtn)\xoo$ as a
stochastic process of trees in discrete time.
It is, as often in similar problems, 
useful to consider also the corresponding process in
continuous time, with items arriving according to a Poisson process with
rate 1.
This means that item $W_n$ arrives at a random time $\taux n$, where the
waiting times $\eta_n:=\taun-\taux{n-1}$ (with $\taux0=0$)
are \iid{} $\Exp(1)$.

\begin{definition}\label{Dctt1}
Let the sequence $(W_n)_n^\infty$ of random items 
arrive according to a Poisson process with rate 1 on $\ooo$.
(As above, the strings $W_n$ are independent, with independent $\Be(1/2)$ bits.)
The continuous-time digital search tree $\ctt$ is the digital search tree  
constructed from the items $W_i$ that have arrived until time $t$.
\end{definition}

Equivalently, we can use \refD{Ddtn}, adding new nodes at times
given by a Poisson process.

Let $N(t)$
be the number of items that have arrived up to time $t$; thus
$N(t)\sim\Po(t)$, and $\ctt$ is the random digital search tree constructed
from a random number $N(t)$ items.
More precisely, the discrete and continuous-time processes $(\dt_n)_n$ and
$(\ctt)_t$ are related by
\begin{align}\label{dtct}
  \ctt&=\dt_{N(t)}
\quad(t\ge0),
&
\dt_n&=\ct_{\taux n}
\quad(n\ge0).
\end{align}
In other words, $\ctt$ is obtained from  $\dtn$ by Poissonization.

Note that $\taun$ is the stopping time when the size $|\ctt|$ becomes $n$.

\subsection{The border aggregation model}\label{SSbam}
\emph{Border aggregation models} on finite connected graphs were studied by
\citet{ThackerVolkov}.  
In general, consider any finite, connected graph with a fixed
vertex $o$,  the \emph{origin}, and a non-empty \emph{boundary} set
denoted by $B$. 
As in the introduction,
we recursively define a randomly growing sequence of
sets of \emph{sticky}  vertices $S_n$ as follows.
\begin{definition}
  \label{Dsticky}
Construct random \emph{sticky sets} $S_n$, $n\ge0$,
as follows.
\begin{romenumerate}
\item 
$S_0=B$, the given boundary.
\item
At times $n=1,2,\dots$,
given $S_{n-1}$, 
let a particle start at $o$ and perform 
some sort of random walk 
until it reaches a neighbour $v_n$ of the sticky set. Then it stops, 
and the node $v_n$
is added to the sticky set, \ie,  $S_{n+1}:=S_{n-1}\cup\set{v_n}$.
\item This is repeated until some time $\xi_K$ when the root becomes sticky;
  then the process  stops.
\end{romenumerate}
\end{definition}
Thus, $\xi_K$ is the number of particles required to build a path from the
boundary to the origin by this aggregation process.
We are (as \citet{ThackerVolkov}) interested in the distribution of $\xi_K$.

This model was introduced as \textit{internal erosion} by  \citet{Le_Pe_2007}.  
In the present paper we consider only the case described in the introduction, 
when the graph is the binary tree $T_K$ and the random walk is the directed
random walk in \refS{SSrandomwalk}.

We use also a continuous-time version of the \bam.

\begin{definition}\label{Dcsticky}
  The \cbam{}  is defined as in 
\refD{Dsticky}, but with particles arriving according to a Poisson process
with rate 1. Let $\cxi_K$ be the time this process stops.
(We assume that the random walk takes no time.)
\end{definition}

Thus, with the notation in \refS{SScdst},
\begin{equation}\label{xixi}
  \cxi_K = \tau(\xi_K)=\sum_{i=1}^{\xi_K} \eta_i,
\end{equation}
where $\eta_i\sim\Exp(1)$ are \iid{} and independent of $\xi_K$.
In particular,
\begin{equation}\label{exixi}
  \E\cxi_K=\E\xi_k.
\end{equation}

\section{More on \cdst{s}}\label{Scdst+}

We give first an alternative construction of the continuous-time digital
search tree $\ctt$ and then show that it agrees with \refD{Dctt1}.

\begin{definition}\label{Dctt2}
Equip each node $v$ in the infinite binary tree $\too$ with a random
variable $X_v\sim\Exp(2^{-d(v)})$, with all $X_v$ independent.
Let 
\begin{equation}\label{Yv}
  Y_v:=\sum_{w\preceq v} X_w,
\qquad v\in \Voo,
\end{equation}
and let $\ctt$ be the extended binary tree with
\begin{equation}\label{dctt2}
  \Vi(\ctt):=\set{v\in \Voo:Y_v\le t}.
\end{equation}
\end{definition}

\begin{remark}\label{Rfpp}
We may interpret the internal nodes in $\ctt$ 
as infected; then \refD{Dctt2} 
describes an infection
that spreads randomly on $\too$ from parents to children,
starting with the root $o$ being infected from the outside, 
where $X_v$ is the time it takes for node $v$ to become infected
once its parent is. (Imagine the root having an outside parent that is
infected at time 0.)
In other words,
$\ctt$
can be seen as first-passage percolation on $\too$,
but note that different edges have different distributions of the 
infection times $X_v$.  
\end{remark}

To see the equivalence of the two definitions, we introduce a third, and
then show that all three are equivalent.

\begin{definition}\label{Dctt3}
Equip each node $v\in\too$ with an
exponential clock that rings with rate $2^{-d(v)}$, independently of all
other clocks.  
Start with $\ct_0$ empty.
Ignore all clocks that are not currently in an external node.
When a clock rings in an external node $v$, then $v$ becomes an internal node
of $\ctt$ and its two children become new external nodes.
\end{definition}

\begin{remark}  \label{RAldousShields}
More generally, \citet{Al_Sh_1988} studied a process defined as in
\refD{Dctt3} but with rates $c^{-d(v)}$ for some constant $c>1$.
(See \cite{Ba_Pe_Pe_1997} for $c<1$.) 
They noted that this is equivalent to \refD{Dctt2} (with these rates),
and that that the process is a random time change of the corresponding
discrete-time process defined as in \refD{Ddtn}, but using instead of the
harmonic measure \eqref{harmonic} on the external nodes the measure where
$p_v$ is proportional to $c^{-d(v)}$.
Note that the simple relation \eqref{harmsum} is special for the case $c=2$,
and thus the relation between the discrete and continuous-time models is in
general more complicated than in \refD{Dctt1}.
\end{remark}

\begin{theorem}[Essentially \citet{Al_Sh_1988}]\label{T=}
   Definitions \ref{Dctt1}, \ref{Dctt2} and \ref{Dctt3}
define the same stochastic process of trees $(\ctt)_{t\ge0}$.
(In the sense of all having the same distribution.)
\end{theorem}

\begin{proof}
In \refD{Dctt3}, the total rate of the clocks in the external nodes is
always 1, by \eqref{harmsum}. Hence, new internal nodes are created with
rate 1.
Furthermore,
if $v$ is an external node, then the clock at $v$ rings with rate
$2^{-d(v)}$, and thus the probability that the clock at $v$ is the next
clock in an external node that rings is also $2^{-d(v)}$.
In other words, when a new internal node is added, it is chosen randomly
among the existing external nodes according to the harmonic measure
\eqref{harmonic}, just as in \refD{Ddtn}.
Hence the process $(\ctt)$ constructed in \refD{Dctt3} has the same
distribution as the one defined in Definitions \ref{Ddtn0}--\ref{Dctt1}.

Furthermore, in \refD{Dctt3}, consider for each node $v\in\too$ the 
stopping time,
$\tau_v$ say,
when $v$ becomes an external node, and let $X_v$ be the waiting
time until the next time the clock at $v$ rings. Then 
$X_v$, $v\in\too$, are independent exponential random variables with the
rates in \refD{Dctt2}. Furthermore, since $\tau_v$ is the time the parent 
of
$v$ becomes an internal node (with $\tau_0=0$ for the root), 
it follows by induction that the time $\tau_v+X_v$ when the clock rings and
$v$ becomes an internal node equals $Y_v$ defined in \eqref{Yv}, and thus
\eqref{dctt2} holds and
the process $(\ctt)_t$ coincides with the one defined by  \refD{Dctt2}.
\end{proof}

In particular, this gives a description of the height $\he(\ctt)$ of $\ctt$,
and thus indirectly also of $\he(\dtn)$.
Use \refD{Dctt2} and let, for $k\ge0$, 
\begin{equation}
  \label{YY}
\YY_k:=\min_{v:\,d(v)=k}Y_v.
\end{equation}
In other words, $\YY_k$ is the smallest sum $\sum_w X_w$ along a path from the
root to a node of depth $k$; 
in the language of \refR{Rfpp},
$\YY_k$ is the time the infection reaches depth $k$.
(I.e., it reaches external height $k+1$.)

\begin{corollary}\label{CA}
We have the equality in distribution, for all $t\ge0$,
\begin{equation}\label{ca}
  \he(\ctt)\eqd 
\min\bigset{k\ge0: \YY_k> t}
=\max\bigset{k\ge0: \YY_k\le t}+1.
\end{equation}
Equivalently, for any $t\ge0$ and $k\ge0$,
\begin{equation}\label{ca2}
  \P\bigpar{\he(\ctt)> k}=\P\bigpar{\YY_k\le t}.
\end{equation}
\end{corollary}
\begin{proof}
  \refD{Dctt2} and \eqref{YY} yield the relation, for $k\ge0$,
  \begin{equation}\label{hy}
\set{\he(\ctt)\le k} 
=\set{v\notin \Vi(\ctt) \text{ when $d(v)=k$}}
=\set{\YY_k> t}.
  \end{equation}
Hence, using \refD{Dctt2}, \eqref{ca} holds with actual equality of the
random variables. By \refT{T=}, we have equality in distribution for any
of the definitions.
\end{proof}


\section{Connection with the border aggregation model}
\label{STB}
\begin{theorem}\label{TB}
  For any $K\ge0$,
$\cxi_{K+1}\eqd\YY_K$.
\end{theorem}

We give two proofs of this theorem. The first uses a simple induction.
The second is longer but perhaps gives more insight; it is more
combinatorial and is based on a study of the aggregation process.
The second proof also provides a coupling of the two processes.

\begin{proof}[First proof of \refT{TB}]
  The claim is trivially true for $K=0$: $\cxi_1$ is the time of arrival of
  the first particle, so $\cxi_1\sim\Exp(1)$ and
$\cxi_1\eqd X_o=\YY_0$.

Denote the two children of the root by $\oL$ and $\oR$.
Consider the \cbam{} on $T_{K+1}$, 
and let $\xxi$ be the time
$\oL$ or $\oR$ becomes sticky.
Then the next particle stops at the root, and thus 
\begin{equation}
  \label{xi1}
\xi_{K+1}=\xxi+X, 
\end{equation}
where
$X\sim\Exp(1)$ is independent of $\xxi$.

Up to time $\xxi$, the particles
proceed to $\oL$ or $\oR$, with
probability $1/2$ each and independently of each other and of the arrival
times of  the particles.
By a standard property of Poisson processes, 
this means that $\oL$ and $\oR$ are fed particles by two independent Poisson
processes with rates $1/2$. 
Let both these processes continue beyond $\xxi$, and let 
$\xil$ and $\xir$ be the times $\oL$ and $\oR$, respectively, then become
sticky. Then 
\begin{equation}\label{xi2}
\xxi=\xil\land\xir.  
\end{equation}
Moreover, the two processes beneath $\oL$ and $\oR$ are independent copies
of the original process on the smaller tree $T_K$, with time running at half
speed. Hence, $\xil\eqd\xir\eqd 2\cxi_K$, and thus by \eqref{xi1} and
\eqref{xi2}, 
\begin{equation}
  \label{xi3}
\cxi_{K+1}\eqd 2\bigpar{\cxi_K\land\cxi_K'}+X
\end{equation}
with $\cxi_K'\eqd\cxi_K$, $X\sim\Exp(1)$ and $\cxi_K$, $\cxi'_K$, $X$
independent.  

Similarly, recalling the definition \eqref{YY} of $\YY_K$, 
let $\YYL$ and $\YYR$ be the smallest sum $\sum X_v$ along a path from $\oL$
or $\oR$, respectively, to a node of depth $K$. Then
\begin{equation}\label{yy1}
  \YY_K=(\YYL+X_o)\land(\YYR+X_o)
= \YYL\land\YYR+X_o.
\end{equation}
Moreover, $\YYL$ and $\YYR$ are independent and both have the same
distribution as $2\YY_{K-1}$, since the subtree of descendants of $\oL$ (or
$\oR$), equipped with their $X_v$ 
is isomorphic to the full tree with root $o$, but given the variables $2X_v$.
Hence, \eqref{yy1} yields
\begin{equation}\label{yy2}
  \YY_K\eqd 2(\YY_{K-1}\land \YYY_{K-1})+X_o,
\end{equation}
with $\YYY_{K-1}\eqd\YY_{K-1}$, 
$X_o\sim\Exp(1)$,
and $\YY_{K-1}$, $\YYY_{K-1}$ and $X_o$ independent.

Comparing \eqref{xi3} and \eqref{yy2}, we see that 
the distributions of
$\cxi_{K+1}$ and $\YY_K$ satisfy the same recursive equation, and thus they
are equal by induction.
\end{proof}

\begin{proof}[Second proof of \refT{TB}]
In the (discrete or continuous-time) \bam, define, at
any given time $t$, the 
\emph{absorption set} $A_t$ as the set of all internal
nodes $v$ such that $v$ is a neighbour of the sticky set $S_t$, but no
ancestor of $v$ is.
Consider only the process $(A_t)$ of absorption sets;
$A_t$ evolves by letting a new particle perform the random walk until it
hits $A_t$, say at $v$. Then $v$ becomes sticky, which means that
the parent $v'$ of $v$ is added to $A_t$, while $v$ and all other
descendants of $v'$ are removed. (If $v$ is the root, then instead the
process stops.)

Note that the absorption set $A_t$ is a boundary in the sense of
\refSS{SSboundary}, and that
given the boundary $A_t$ at some time $t$, the next node that becomes sticky
is chosen randomly from $A_t$ according to the harmonic measure on $A_t$,
see Sections \ref{SSrandomwalk} and \ref{SSboundary}.
Furthermore, $(A_t)_t$ is a Markov process.

From now on we consider the continuous-time version; furthermore, 
we consider the tree $T_{K+1}$ with external nodes at depth $K+1$.
Equip the nodes $v\in V_K:=\set{v:d(v)\le K}=\Vi(T_{K+1})$
with exponential clocks as in \refD{Dctt3}.
Define a process $A_t'$ of subsets of $V_K$ as follows:
\begin{romenumerate}
\item 
$A_0':=A_0=\set{v:d(v)=K}$.
\item 
Clocks outside the current $A_t'$ are ignored.
When a clock at a node $v\in A_t'$ rings, $A_t'$ is updated as above;
i.e., 
the parent $v'$ of $v$ is added to $A_t'$, while $v$ and all other
descendants of $v'$ are removed. (If $v$ is the root, then instead the
process stops.)
\end{romenumerate}
Given $A_t'$, the next clock in $A_t'$ that rings is random with a
distribution given by the harmonic measure on $A_t'$.
Hence, the process $A_t'$ just constructed has the same distribution as the
process $A_t$ in the aggregation process, and we may assume that $A_t=A_t'$
for all $t\ge0$.

For each node $v\in V_K$, let now 
$\tau_v:=\inf\set{t\ge0:u\in A_t \text{ for some }u\preceq v}$, 
i.e.,  the first time that either $v$ or one of its ancestors belongs to the
absorption set, and let
$X_v$ be the waiting time from $\tau_v$ to the next time that the clock at
$v$ rings. Then the random variables $X_v$, $v\in V_K$, are independent 
and have the exponential distributions given in \refD{Dctt2}. 
(We may define $X_v$ also for
$d(v)>K$ for completeness, but these variables will not matter.)
Define $Y_v$ by \eqref{Yv}.

For a node $v\in V_K$, let
\begin{equation}\label{zv}
  Z_v:=\min_{w\succeq v,\, d(w)=K}\bigset{Y_w-Y_v}.
\end{equation}
This is the minimum over the paths from $v$ to the boundary $\Ve(T_{K+1})$
of the sum  $\sum_u X_u$ for all nodes $u$ in the path, excluding
the endpoints. In particular, $Z_v=0$ when $d(v)=K$.

We claim that at any time $t\ge0$ with $t\le\cxi_{K+1}$,
\begin{equation}\label{at+}
  A_t=\bigset{v\in V_K:Z_v\le t\text{ but } Z_u>t \text{ for all }u\prec v},
\end{equation}
and furthermore 
\begin{align}\label{at++}
\tau_v=Z_v     \qquad \text{for every $v\in A_t$}.
\end{align}
We prove this claim by induction; it is evidently true for $t=0$, and it
then suffices to consider the finite number of times that $A_t$ changes.

Suppose that the claim holds for some time $t$. 
If $v\in A_t$, then the next time that the clock at $v$ rings is,
letting again $v'$ be the parent of $v$ and noting that
$Y_v=Y_{v'}+X_v$ 
(with $Y_{o'}:=0$),
\begin{equation}\label{zw}
\tau_v+X_v=Z_v+X_v  
=\min_{w\succeq v,\, d(w)=K}\bigset{Y_w-Y_{v'}}.
\end{equation}
Let $v$ be the node in 
the current $A_t$ such that the time $Z_v+X_v$ in \eqref{zw}
is minimal. 
Then $v$ is the next node to become sticky, and its parent $v'$ is 
the next node added to $A_t$;
this happens at time $\tau_{v'}=Z_v+X_v$,
which by \eqref{zw} equals the minimum 
over all paths from $v'$ to $\Ve(T_{K+1})$ that pass through $v$
of the sum $\sum_uX_u$ for $u$ in the path, excluding the endpoints.
A path from $v'$ to $\Ve(T_{K+1})$ that does not pass through $v$ must pass
through some other node $v''\in A_t$, and since $Z_{v''}+X_{v''}\ge
Z_v+X_v$,
it follows that $\sum_uX_u$ for $u$ in this path is $\ge Z_v+X_v$.
Hence, using \eqref{zw} and \eqref{zv}, $\tau_{v'}=Z_v+X_v= Z_{v'}$;
moreover \eqref{at+} holds up to time $Z_v+X_v$.
This completes the induction step, and thus the proof of the claim
\eqref{at+}--\eqref{at++}.

Obviously, $o\in A_t$ for some $t$, and thus \eqref{at++} applies to $v=o$.
Consequently,
the time $\cxi_{K+1}$ that the root becomes sticky is, 
using the definitions of $\tau_o$ and $X_o$ together with
\eqref{at++},
\eqref{Yv}, 
\eqref{zv}
and \eqref{YY}, 
\begin{equation}
\cxi_{K+1}=
  \tau_o+X_o
=Z_o+X_o
=Z_o+Y_o
=\YY_K.
\end{equation}
\end{proof}

\begin{proof}[Proof of \refT{TC}]

\pfitemref{Tcc}
\refT{TB} and \refC{CA} yield, for $K\ge1$ and $t\ge0$,
\begin{equation}
  \P\bigpar{\cxi_{K}\le t}
=
  \P\bigpar{\YY_{K-1}\le t}
=
\P\bigpar{\he(\ctt)\ge K}.
\end{equation}

\pfitemref{Tcd}  
By \eqref{hy},
\begin{equation}
  \YY_k = \min\bigset{t\ge0: \he(\ctt)> k}.
\end{equation}
Define analoguously, for the discrete time process,
\begin{equation}\label{YYX}
  \YYX_k := \min\bigset{n\ge0: \he(\dtn)> k}.
\end{equation}
Then, see the relations \eqref{dtct},
\begin{equation}\label{metro}
  \YY_k = \taux{\YYX_k}
=\sum_{i=1}^{\YYX_k}\eta_i,
\end{equation}
where as in \eqref{xixi}, $\eta_i$ are \iid{} $\Exp(1)$ and independent of
the discrete time process.
Hence, \eqref{xixi}, \refT{TB} and \eqref{metro} yield
\begin{equation}\label{rgn}
  \sum_{i=1}^{\xi_{K+1}}\eta_i
=\Xi_{K+1}
\eqd
\YY_K=
\sum_{i=1}^{\YYX_K}\eta_i.
\end{equation}
If we take the Laplace transforms of the \lhs, we obtain
by conditioning on $\xi_{K+1}$, for any $s\ge0$,
\begin{equation}
  \E\exp\biggpar{-s \sum_{i=1}^{\xi_{K+1}}\eta_i}
=\E \Bigpar{\bigpar{\E e^{-s\eta}}^{\xi_{K+1}}}
=\E \Bigpar{\xpar{1+s}^{-\xi_{K+1}}}.
\end{equation}
This and an identical calculation for the \rhs{} show that,
taking $s=x\qw-1$,
$ \E\bigpar{x^{\xi_{K+1}}}=\E\bigpar{x^{\YYX_K}}$ for every $x\in(0,1)$.
In other words, $\xi_{K+1}$ and $\YYX_K$ have the same \pgf, and thus the
same distribution.

Consequently, 
using the definition \eqref{YYX}, for $K\ge0$,
\begin{equation}
  \P\bigpar{\xi_{K+1}\le n}
=
  \P\bigpar{\YYX_{K}\le n}
=
  \P\bigpar{\he(\dtn)>K}
=
  \P\bigpar{\he(\dtn)\ge K-1}
.\end{equation}
The result follows by replacing $K$ by $K-1$.
\end{proof}

\section{Proofs of \refTs{Txi} and \ref{TE}}\label{SpfTxi}

We next prove \refT{Txi},
using
\citet[Theorem 4 and its proof in Section 6.1]
{DrmotaEtAl-DST}.

\begin{proof}[Proof of \refT{Txi}]
Let $n=n_K$, $K\ge1$, be such that
\begin{align}\label{ja}
\lgg n=
K -\sqrt{2K} + \frac12\log_2 K -\frac{1}{\ln2} +
\frac{\log_2K}{4\sqrt{2K}} + \frac{a_K}{\sqrt K}
\end{align}
for some sequence $a_K=\oo\bigpar{\sqrt K}$ as \Ktoo.
Define
\begin{align}\label{jb}
  \kkk&:=\lgg n + \sqrt{2\lgg n} -\frac12\lgg\lgg n + \frac{1}{\ln2},
\\\label{hth}
\hth&:=\frac{3\lgg\lgg n}{4\sqrt{2\lgg n}},
     \end{align}
{and, as in \cite{DrmotaEtAl-DST},}
     \begin{align}\label{jc}  
k_H&:=\floor{\kkk},
\\\label{jd}
k_\ell&:=k_H+\ell, \text{ for } \ell \in \bbZ, 
\\\label{je}
\gth&:=\kkk-k_H\in[0,1),
\end{align}
Elementary calculations show that 
\begin{align}\label{jf}
  \sqrt{\lgg n} &
= \sqrt{K} -\frac{1}{\sqrt2} + \frac{\lgg K}{4\sqrt K} +
  \OO\Bigparfrac{1}{\sqrt K},
\\\label{jg}
  \lgg\lgg n &
= \lgg{K} +  \OO\Bigparfrac{1}{\sqrt K},
\\\label{jh}
\kkk-\hth
&=
K -1+ \frac{a_K}{\sqrt K}+ \OO\Bigparfrac{1}{\sqrt K},
\\\label{jk}
k_1&=\kkk-\gth+1
=K+ \hth-\gth + \frac{a_K+\OO(1)}{\sqrt K} 
,\\\label{jl}
\gth-\hth&
=K-k_1 + \frac{a_K+\OO(1)}{\sqrt K} 
.\end{align}
In particular, since $a_K=\oo\bigpar{\sqrt{K}}$ and $\hth=\oo(1)$,
\eqref{jh} implies
\begin{align}\label{jm}
\kkk=K-1+\oo(1)
,\end{align}
and thus, for all large $K$,
\begin{align}\label{jn}
k_H=\floor{\kkk}\in\set{K-1,K-2}.  
\end{align}
In other words, for large $K$, either $K=k_H+1=k_1$ or $K=k_H+2=k_2$. 

Suppose now that $a_K\to-\infty$. 
On the subsequence where $K=k_2$ (if there are any such $K$),
we have by \eqref{tcd} and \cite[Lemma 14]{DrmotaEtAl-DST}, 
writing $H_n:=\he(\dtn)$ as in  \cite{DrmotaEtAl-DST}, 
\begin{align}\label{jo}
  \P(\xi_k \le n)
=\P\bigpar{H_n\ge K}
=\P\bigpar{H_n> k_H+1}\to0.
\end{align}
On the subsequence where $K=k_1$ (if there are any such $K$),
\eqref{jl} yields
\begin{align}\label{jp}
  \gth-\hth = -\frac{\go(1)}{\sqrt K}
= -\frac{\go(1)}{\sqrt{\lgg n}},
\end{align}
and thus, using also
\cite[Remark 5]{DrmotaEtAl-DST}, 
$\P(H_n=k_1)\to0$, and thus
\begin{align}\label{jq}
  \P(\xi_k \le n)
=\P\bigpar{H_n\ge K}
=\P\bigpar{H_n= k_H+1}
+ \P\bigpar{H_n> k_H+1}
\to0.
\end{align}
Together, \eqref{jo} and \eqref{jq} show that if $a_K\to-\infty$,
then $  \P(\xi_k \le n)\to0$ as \Ktoo, regardless of whether $K=k_1$ or $k_2$.

On the other hand, suppose that $a_K\to+\infty$.
Since $\hth>0$ (for large $K$ at least), \eqref{jh} implies that for large
$K$,
$\kkk\ge K-1$, and thus, by \eqref{jn}, $k_H=K-1$ and $K=k_1$.
Furthermore, \eqref{jl} implies 
\begin{align}\label{jpp}
  \gth-\hth = \frac{\go(1)}{\sqrt K}
= \frac{\go(1)}{\sqrt{\lgg n}}.
\end{align}
Hence, 
\cite[Remark 5 and Lemma 13]{DrmotaEtAl-DST} imply that 
$\P(H_n\le k_H)\to0$, and
thus \eqref{tcd} yields
\begin{align}\label{jr}
  \P(\xi_k \le n)
=\P\bigpar{H_n\ge K}
=\P\bigpar{H_n>k_H}
\to1.  
\end{align}

Finally, define
  \begin{equation}\label{txiZ}
Z_K:=\sqrt K\Bigpar{\log_2  \xi_K - 
\Bigpar{K -\sqrt{2K} + \frac12\log_2 K -\frac{1}{\ln 2} +
\frac{\log_2K}{4\sqrt{2K}}}}.
  \end{equation}
Then, \eqref{jo}, \eqref{jq} and \eqref{jr} 
show, together with \eqref{ja}, that if $a_K\to-\infty$, then 
$\P(Z_K\le a_K)\to0$, while 
if $a_K\to+\infty$, then 
$\P(Z_K\le a_K)\to1$.
This is equivalent to $Z_K=\Op(1)$, and thus to \eqref{txi}.
\end{proof}

Finally, we use \refT{Txi} to prove \refT{TE} on the mean.

\begin{proof}[Proof of \refT{TE}]
In this proof, all limits are as \Ktoo.
First,
\eqref{xik2} implies, 
\begin{align}\label{ptoxi}
\xi_K/\mk\pto1, 
\end{align}
and thus, by \eqref{xixi} and the law of large numbers,
\begin{align}\label{ptoXi}
\Xi_K/\mk\pto1.
\end{align}
Note that
this immediately implies, by Fatou's lemma
\cite[Theorem 5.5.3]{Gut},
\begin{align}\label{eleonora}
  \liminf_{\Ktoo}\frac{\E\Xi_K}{\mk}
\ge 1.
\end{align}

To obtain also an upper bound,
we use tail estimates by \citet{Drmota2002}. Note that Drmota uses the
internal height, thus his $H_n=\he(\dtn)-1$. 
Furthermore, $P_k(x)$ in \cite{Drmota2002} is the distribution function of the 
Poissonized version of $H_n$, and thus in our notation
\begin{align}
  P_k(x)=\P\bigpar{\he(\ct_x)-1 \le k}.
\end{align}
Hence, by \eqref{tcc}, for $K\ge 2$ and $x\ge0$,
\begin{align}\label{kucku}
  \P(\Xi_K>x) = \P\bigpar{\he(\ct_x)\le K-1}=P_{K-2}(x).
\end{align}
We use \cite[Lemma 4]{Drmota2002}, 
for convenience denoting $n_{K-2}$ there
by $\nk$ and noting that $\frac12< c_k< 1$ for large $k$;
this yields together with \eqref{kucku}, for large $K$,
\begin{align}
  \P(\Xi_K>x) &\ge \Bigpar{1-\frac{1}{\nk}}e^{-x/\nk},
&& 0\le x\le \nk, \label{drm<}
\\
  \P(\Xi_K>x) &\le e^{-x/(2\nk)},
&&  x\ge \nk.\label{drm>}
\end{align}
Let $\eps>0$. Then \eqref{ptoXi} says that 
$\P\bigpar{(1-\eps)\mk<\Xi_K<(1+\eps)\mk}\to1$,
which combined with \eqref{drm<}--\eqref{drm>} (taking $x=\nk$)
implies that for large $K$
we must have $(1-\eps)\mk < \nk<(1+\eps)\mk$.
In other words,
\begin{align}\label{nkmk}
  \nk/\mk\to1.  
\end{align}
Hence, \eqref{ptoXi} is equivalent to $\Xi_K/\nk\pto1$,
which means 
\begin{align}
\P\bigpar{\xfrac{\Xi_K}{\nk}>x}\to\ett{x<1} 
\end{align}
for every $x\neq1$.  
Furthermore, \eqref{drm>} implies that, for large $K$,
\begin{align}
\P\bigpar{\xfrac{\Xi_K}{\nk}>x}\le \ett{x<1}+e^{-x/2}   
\end{align}
for every $x\ge0$.
Consequently, dominated convergence yields
\begin{align}
  \E\frac{\Xi_K}{\nk}
=\intoo \P\Bigpar{\frac{\Xi_K}{\nk}>x}\dd x
\to\intoo \ett{x<1} \dd x=1
\end{align}
as \Ktoo.
The result follows by \eqref{nkmk} and \eqref{exixi}.
\end{proof}

\begin{remark} \label{RE}
To prove \refConj{CONJE} by similar arguments, one would need much stronger
tail estimates that \eqref{drm<}--\eqref{drm>}.
It seems that the method of proof of \cite[Lemma 14]{DrmotaEtAl-DST} 
might give the required estimates;
however, we have not verified the (non-trivial) 
details and leave the conjecture as an open problem.
%
\end{remark}

\section{$b$-ary trees}\label{Sdary}

We have in this paper only considered binary trees.
A random $b$-ary  digital search tree can be constructed 
in the same way for any given $b\ge2$, using strings 
$W_i$ with letters from an alphabet $\cA$ of size $b$, for example
$\cA=\{0,1,\cdots, b-1\}$; we still assume that the letters are independent
and that all letters have the same probability (\viz{} $1/b$).

Similarly,  the \bam{} can be defined on $b$-ary trees as in Definition
\ref{Dsticky}, where now the  random walk at each step selects a child with
probability $1/b$ each.

Most of the results above hold with only trivial changes.
The  harmonic measure \eqref{harmonic} becomes $b^{-d(v)}$.
In Definitions \ref{Dctt2} and \ref{Dctt3}, the rate should be
$b^{-d(v)}$.
In particular, \refTs{TC} and \ref{TB} still hold (by the same arguments).

However, \refT{Txi} uses results for the binary case proved in
 \cite{DrmotaEtAl-DST}; the results and methods there ought to generalize
to arbitrary $b$, but that has not yet been done, so we cannot extend this
result to larger $b$.
Nevertheless,
we conjecture that 
for the \bam{} on regular $b$-ary trees, 
for a suitable constant $c_b>0$, 
\begin{equation}
\label{Eq: number_particles_d-ary_tree_high_prob}
\log_b  \xi_K = K -\sqrt{2K}+c_b \log_b K+\Op(1). 
\end{equation}


\newcommand\AAP{\emph{Adv. Appl. Probab.} }
\newcommand\JAP{\emph{J. Appl. Probab.} }
\newcommand\JAMS{\emph{J. \AMS} }
\newcommand\MAMS{\emph{Memoirs \AMS} }
\newcommand\PAMS{\emph{Proc. \AMS} }
\newcommand\TAMS{\emph{Trans. \AMS} }
\newcommand\AnnMS{\emph{Ann. Math. Statist.} }
\newcommand\AnnPr{\emph{Ann. Probab.} }
\newcommand\CPC{\emph{Combin. Probab. Comput.} }
\newcommand\JMAA{\emph{J. Math. Anal. Appl.} }
\newcommand\RSA{\emph{Random Struct. Alg.} }
\newcommand\ZW{\emph{Z. Wahrsch. Verw. Gebiete} }
\newcommand\DMTCS{\jour{Discr. Math. Theor. Comput. Sci.} }

\newcommand\AMS{Amer. Math. Soc.}
\newcommand\Springer{Springer-Verlag}
\newcommand\Wiley{Wiley}

\newcommand\vol{\textbf}
\newcommand\jour{\emph}
\newcommand\book{\emph}
\newcommand\inbook{\emph}
\def\no#1#2,{\unskip#2, no. #1,} 
\newcommand\toappear{\unskip, to appear}

\newcommand\arxiv[1]{\texttt{arXiv}:#1}
\newcommand\arXiv{\arxiv}

\def\nobibitem#1\par{}


\begin{thebibliography}{99}

\bibitem[Aldous and Shields(1988)]{Al_Sh_1988} 
David Aldous and Paul  Shields. 
A diffusion limit for a class of randomly-growing binary trees.
\emph{Probab. Theory Related Fields} \vol{79} (1988), no. 4, 509--542.

\bibitem[Barlow, Pemantle, Perkins(1997)]{Ba_Pe_Pe_1997} 
Martin T Barlow,   Robin Pemantle and  Edwin A. Perkins.
Diffusion-limited aggregation on a tree.
\emph{Probab. Theory Related Fields} \vol{107} (1997), no. 1, 1--60.

\bibitem[Drmota(2002)]{Drmota2002}
Michael Drmota. 
The variance of the height of digital search trees. 
\emph{Acta Inform.} \vol{38} (2002), no. 4, 261--276.

\bibitem[Drmota(2009)]{Drmota-book}
Michael Drmota,
\book{Random Trees},
Springer, Vienna, 2009.

\bibitem[Drmota, Fuchs,  Hwang and  Neininger(2017+)]{DrmotaEtAl-DST}
Michael Drmota, Michael Fuchs, Hsien-Kuei Hwang and Ralph Neininger, 
Node profiles of symmetric digital search trees.
Preprint, 2020.
\arxiv{1711.06941v4}

\bibitem[Drmota, Janson, Neininger(2008)]{Dr_Ja_Ne_2008} 
Michael Drmota,  Svante Janson and Ralph Neininger. 
A functional limit theorem for the profile of search trees. 
\emph{Ann. Appl. Probab.} \vol{18} (2008), no. 1, 288--333.

\bibitem{Gut}
Allan Gut.
\emph{Probability: A Graduate Course}.
Springer, New York, 2005. 

\bibitem[Jacquet and  Szpankowski(2015)]{JacquetS-book}
Philippe Jacquet and Wojciech Szpankowski:
\emph{Analytic Pattern Matching: From DNA to Twitter}.
Cambridge University Press, Cambridge, 2015.

\bibitem[Kesten(1987)]{Ke_1987}
Harry Kesten.
How long are the arms in DLA?
\emph{J. Phys. A} \vol{20} (1987), no. 1, L29--L33.

\bibitem[Knessl and Szpankowski(2000)]{Kn_Sz_2000} 
Charles Knessl and Wojciech Szpankowski.
Asymptotic behavior of the height in a digital search tree and the longest
phrase of the Lempel-Ziv scheme.
\emph{SIAM J. Comput.} \vol{30} (2000), no. 3, 923--964.

\bibitem[Knuth(1998)]{KnuthIII} 
Donald E. Knuth:
\emph{The Art of Computer Programming. Vol. 3: Sorting
 and Searching}. 
2nd ed., Addison-Wesley,
Reading, MA, 1998.

\bibitem[Lawler,  Bramson and Griffeath(1992)]{LawlerBG}
Gregory F. Lawler, Maury Bramson and David Griffeath:
Internal diffusion limited aggregation.
\emph{Ann. Probab.} \vol{20} (1992), no. 4, 2117--2140.

\bibitem[Levine and Peres(2007)]{Le_Pe_2007}
Lionel Levine and  Yuval Peres. 
Internal erosion and the exponent $\xfrac{3}{4}$. 
Preprint, 2007.\\
\url{http://www.math.cornell.edu/~levine/erosion.pdf}

\bibitem[Levine and Silvestri(2018)]{Le_Si_2018}
Lionel Levine and Vittoria Silvestri.
How long does it take for internal DLA to forget its initial profile?
\emph{Probab. Theory Related Fields} \vol{174} (2019), no. 3-4, 1219--1271.

\bibitem[Mahmoud(1992)]{Mahmoud-book}
Hosam M. Mahmoud:
\book{Evolution of Random Search Trees}, 
\Wiley, New York, 1992.

\bibitem[Silvestri(2019)]{Si_2019}
Vittoria Silvestri. 
Internal DLA on cylinder graphs: fluctuations and mixing. 
Preprint, 2019.
\arXiv{1909.09893}.

\bibitem[Thacker and Volkov(2017)]{ThackerVolkov}
Debleena Thacker and Stanislav Volkov.
Border aggregation model.
\emph{Ann. Appl. Probab.} \vol{28} (2018), no. 3, 1604--1633.

\bibitem[Witten and Sanders(1983)]{Wi_Sa_1983}
T. A. Witten and L. M. Sander.
Diffusion-limited aggregation.
\emph{Phys. Rev. B (3)} \vol{27} (1983), no. 9, 5686--5697.





\end{thebibliography}
\end{document}